\documentclass[12pt]{amsart}
\usepackage{geometry}
\usepackage{hyperref,amsmath}
\usepackage{tikz-cd}
\usepackage{mathpazo}
\usepackage{amssymb,url,setspace,xcolor}

\usepackage{mathrsfs}

\newtheorem{theorem}{Theorem}

\newtheorem{lemma}[theorem]{Lemma}

\newtheorem{remark}[theorem]{Remark}

\def\arrow{\rightarrow}
\def\C{\mathbb{C}}
\def\Cm{\mathbb{C}^m}
\def\Cn{\mathbb{C}^n}
\def\Linfty{L^{\infty}(\Omega)}
\def\Lo{L^2(\Omega)}
\def\Lao{L^2_a(\Omega)}
\def\Hinfty{{H^\infty(\Omega)}}
\def\Ao{A(\Omega)}
\def\row#1#2{#1_1,\ldots,#1_{#2}}
\def\ol{\overline}
\def\Afb{A[\row {\ol f}m]}
\def\Hfb{\Hinfty[\row {\ol f}m]}
\def\ma{\mathfrak{M}_A}
\def\mb{\mathfrak{M}_B}
\def\maf{\mathfrak{M}_{A[\row {\ol f}m]}}

\def\mh{\mathfrak{M}_{\Hinfty}}

\def\mo{\mathfrak{M}_{\Ao}}
\def\E{{\mathscr E}}

\def\myhat{\widehat}
\def\what{\widehat}
\def\sm{\setminus}

\def\tg{\widetilde g}

\onehalfspace

\begin{document}

\title[A uniform algebra approach to a theorem of Sahuto\u glu and Tikaradze]{A uniform algebra approach\\ to an approximation theorem of Sahuto\u glu and Tikaradze}

\author{Timothy G. Clos}
\address{Department of Mathematical Sciences, Kent State University, Kent, OH 44242}
\email{tclos@kent.edu}
\author{Alexander J. Izzo}
\address{Department of Mathematics and Statistics, Bowling Green State University, Bowling Green, OH 43403}
\email{aizzo@bgsu.edu}
\thanks{The second author was supported by NSF Grant DMS-1856010.}
\date{\today}


 \maketitle

\begin{abstract}
Using methods from the theory of uniform algebras, we give a simple proof of an approximation result of Sahuto\u glu and Tikaradze with $L^\infty$-pseudoconvex domains replaced by the open sets for which Gleason's problem is solvable.
\end{abstract}

\section{The Results}

In \cite{ST} 
S\"onmez Sahuto\u glu and Akaki Tikaradze proved, on what they referred to as $L^\infty$-pseudoconvex domains, an approximation result that can be viewed as a several complex variables generalization of a weak form of an earlier approximation result in one complex variable due to Christopher Bishop\footnote{Christopher Bishop should not be confused with Errett Bishop after whom the antisymmetric decomposition, which will appear later in our paper, is named.} \cite{CBishop}.  They used their approximation result to give a generalization to several complex variables of a theorem of Sheldon Axler, \v Zeljko \v Cu\v ckovi\v c, and Nagisetti Rao regarding commuting Toeplitz operators \cite{acr}.  The main purpose of the present paper is to give a simple proof of the approximation result of Sahuto\u glu and Tikaradze, under a different hypothesis on the underlying domain, using methods from the theory of uniform algebras.

We introduce here some notation and terminology we will use.  Throughout the paper, $\Omega$ will be an open set in $\Cn$ or the in the Riemann sphere.  The boundary of $\Omega$ will be denoted by $b\Omega$.
Following \cite{ST}, 
given a holomorphic map $f:\Omega\arrow \Cm$ we will denote by $\Omega_{f,\lambda}$ the set of all nonisolated points of $f^{-1}(\lambda)$ and we set 
$\Omega_f=\bigcup_{\lambda\in \Cm}\Omega_{f,\lambda}$.
For a compact space $X$,
we denote by
$C(X)$ the algebra of all continuous complex-valued functions on $X$.
A \emph{uniform algebra} on $X$ is a supremum norm closed subalgebra of $C(X)$ that contains the constant functions and separates
the points of $X$.  In particular, a uniform algebra is a commutative Banach algebra.
We will denote the maximal ideal space of a commutative Banach algebra $A$ by $\ma$.  Given $x\in A$ we will denote the Gelfand transform of $x$ as usual by $\myhat x$.  If $A$ is a Banach algebra of continuous complex-valued functions on a subset of $\Cn$ and the complex coordinate functions $\row zn$ belong to $A$, we will let $\pi_A:\ma\arrow\Cn$ denote the map given by $\pi_A(x)=\bigl(\myhat z_1(x),\ldots, \myhat z_n(x)\bigr)$.  
As usual $\Hinfty$ will denote the algebra of bounded holomorphic functions on $\Omega$ equipped with the supremum norm.   If $A$ is an algebra of bounded continuous complex-valued functions on
$\Omega$ and $\row fm$ are bounded continuous complex-valued functions on $\Omega$, we will denote by $A[\row fm]$ the norm closed subalgebra of $\Linfty$ generated by $A$ and the functions $\row fm$.  This last notation, which is rather standard, differs from the notation in \cite{ST} in that in \cite{ST} the notation $A[\row fm]$ is used to denote the algebra generated by $A$ and $\row fm$ without taking  closure.

In the terminology of Sahuto\u glu and Tikaradze, an $L^\infty$-pseudoconvex domain is a pseudoconvex domain on which the $\ol\partial$ problem is solvable in $L^\infty$.
(See \cite{ST} for the precise definition.)
The approximation theorem of Sahuto\u glu and Tikaradze referred to above is the following.

\begin{theorem}[\cite{ST}, Theorem~1]\label{STapprox}
Let $\Omega$ be a bounded $L^\infty$-pseudoconvex domain in $\Cn$ and let $f_j\in\Hinfty$ for $j=1,\ldots, m$.  Set $f=(\row fm)$.  Suppose that $g\in C(\ol \Omega)$ satisfies $g|_{b\Omega\cup \Omega_f}= 0$.  Then $g$ is in $\Hfb$.
\end{theorem}

This theorem can be regarded as a partial extension to several variables of an approximation theorem of Christopher Bishop.

\begin{theorem}[\cite{CBishop}, Theorem~1.2]\label{Chris}
Suppose that $\Omega$ is an open set in the Riemann sphere and that $f\in \Hinfty$ is nonconstant on each component of $\Omega$.  Then $C(\ol\Omega)\subset \Hinfty[\ol f]$.
\end{theorem}

Sahuto\u glu and Tikaradze's proof of Theorem~\ref{STapprox} was inspired by Bishop's proof of Theorem~\ref{Chris}, and like Bishop's proof, it is rather long and complicated.  A simpler proof of Bishop's theorem was given by the second author of the present paper in \cite{Izzo1993} using uniform algebra methods.  Here we will use uniform algebra methods to give a simple proof of Theorem~\ref{STapprox} with the hypothesis that $\Omega$ is an $L^\infty$-pseudoconvex domain replaced by the hypothesis that $\Omega$ is open when regarded as a subset of the maximal ideal space $\mh$ of $\Hinfty$.
(We regard $\Omega$ as a subset of $\mh$ by identifying each point $\lambda$ in $\Omega$ with the functional \lq\lq evaluation at $\lambda$\rq\rq.)
We state the result explicitly here.

\begin{theorem}\label{main1}
Let $\Omega$ be a bounded open set in $\Cn$ such that $\Omega$ is open in $\mh$, and let $f_j\in\Hinfty$ for $j=1,\ldots, m$.  Set $f=(\row fm)$.  Suppose that $g\in C(\ol \Omega)$ satisfies $g|_{b\Omega\cup \Omega_f}= 0$.  Then $g$ is in $\Hfb$.
\end{theorem} 

By exactly the same prove we will also establish the analogous assertion for the algebra $\Ao$ of continuous complex-valued functions on $\ol\Omega$ that are holomorphic on $\Omega$.

\begin{theorem}\label{main2}
Let $\Omega$ be a bounded open set in $\Cn$ such that $\Omega$ is open in $\mo$, and let $f_j\in\Ao$ for $j=1,\ldots, m$.  Set $f=(\row fm)$.  Suppose that $g\in C(\ol \Omega)$ satisfies $g|_{b\Omega\cup \Omega_f}= 0$.  
Then $g$ is in $\Ao[\row {\ol f}m]$.
\end{theorem} 

We will see also that a similar argument in combination with a result in the second author's paper \cite{Izzo2003} yields yet another proof of Theorem~\ref{Chris}.

The class of domains $\Omega$ for which $\Omega$ is open in $\mh$ (or $\mo$) is quite broad.  To see this, note that for $A$ a Banach algebra of continuous complex-valued functions on $\Omega$ containing the functions $\row zn$, the set $\Omega$ is open in $\ma$ 
whenever 
$\pi_A$ is injective over $\Omega$, since in that case $\Omega$ (regarded as a subset of $\ma$) coincides with $\pi_A^{-1}(\Omega)$.  Furthermore this injectivity over $\Omega$ obviously holds whenever Gleason's problem is solvable for $A$, i.e., whenever, for every point $a=(\row an)\in \Omega$, the functions $z_1-a_1,\ldots, z_n-a_n$ generate the ideal of functions in $A$ vanishing at $a$.  Gleason's problem has been extensively studied and is known to be solvable for $\Hinfty$ and $\Ao$ on many classes of domains.  (See for instance, \cite{AS} for the case of strongly pseudoconvex domains, or for the particular case of the ball \cite{rudin}.)

Theorems~\ref{STapprox}, \ref{main1}, and~\ref{main2} can be reformulated using the notion of essential set.  For a uniform algebra $A$ on a compact space $X$, the \emph{essential set} $\E$ for $A$ is the smallest closed subset of $X$ such that $A$ contains every continuous complex-valued function on $X$ that vanishes on $\E$.
The existence of the essential set was proved by 
Herbert Bear \cite{Bear} (or see \cite{Browder}).  
Theorem~\ref{main2} asserts that under the given hypotheses on $\Omega$ and $f$, the essential set for $\Ao[\row {\ol f}m]$ is contained in $b\Omega \cup \ol\Omega_f$.
The conclusion of Theorems~\ref{STapprox} and \ref{main1} can be reformulated as   the assertion that the essential set for $\Hfb$ regarded as a uniform algebra on its maximal ideal space is contained in $(\mh\sm\Omega)\cup \ol\Omega_f$.

One reason for interest in 
the above theorems stems from an application to Toeplitz operators given by Sahuto\u glu and Tikaradze.
Let $\Lao$ denote the Bergman space, i.e., the space of square integrable, holomorphic functions on $\Omega$, and let $P: \Lo\arrow\Lao$ denote the Bergman projection, i.e., the orthogonal projection of $\Lo$ onto $\Lao$.  For $\phi\in \Linfty$ the Toeplitz operator $T_{\phi}:\Lao\rightarrow \Lao$ is defined by the equation $T_{\phi}(f)=P(\phi f)$.  The commuting Toeplitz operator problem is to characterize those functions $\phi, \psi\in \Linfty$ such that $T_\phi$ and $T_\phi$ commute.  
With the Hardy space in place of the Bergman space, the commuting Toeplitz operator problem was solved by Arlen Brown and Paul Halmos in \cite{bh}.  On the Bergman space, the problem is still open even on the disk.  There are, however, various partial solutions including the following result due to Sheldon Axler, \v Zeljko \v Cu\v ckovi\v c, and Nagisetti Rao \cite{acr}.

\begin{theorem}[\cite{acr}]\label{ACR}
Let $\Omega$ be a domain in the complex plane, let
$\phi$ be a nonconstant bounded holomorphic function on $\Omega$, and let $\psi$ is a bounded measurable function on $\Omega$ such that
$T_\phi$ and $T_\psi$ commute.  Then $\psi $ is holomorphic.  
\end{theorem}

Axler, Cu\v ckovi\v c, and Rao obtained this theorem as a consequence of Theorem~\ref{Chris} of Bishop.  Sahuto\u glu and Tikaradze used their partial extension of Bishop's theorem to several variables (Theorem~\ref{STapprox} above),  to give a generalization of the Axler-Cu\v ckovi\v c-Rao theorem to several variables. 

\begin{theorem}[\cite{ST}, Corollary~2]\label{ST2}
Let $\Omega$ be a bounded $L^\infty$-pseudoconvex domain in $\Cn$, let $g\in L^\infty(\Omega)$, and let $f_j\in H^{\infty}(\Omega)$ for all $j=1,\ldots, m$.  Suppose the Jacobian of the map
$f=(\row fm):\Omega\arrow \Cm$ has rank $n$ at some point $z\in \Omega$ and $T_g$ commutes with $T_{f_j}$ for all $1\leq j\leq m$.  Then $g$ is holomorphic. 
\end{theorem}

As an intermediate step in the proof of Theorem~\ref{ST2}, Sahuto\u glu and Tikaradze used Theorem~\ref{STapprox} to prove an $L^p$-approximation theorem.

\begin{theorem}[\cite{ST}, Corollary~1]\label{ST1}
Let $\Omega$ be a bounded $L^\infty$-pseudoconvex domain in $\Cn$ and $f_j\in H^{\infty}(\Omega)$ for all $j=1,\ldots, m$.  Then the following are equivalent.
\item{\rm(i)} $\Hfb$ is dense in $L^p(\Omega)$ for all $0<p<\infty$.
\item{\rm(ii)} $\Hfb$ is dense in $L^p(\Omega)$ for some $1\leq p<\infty$.
\item{\rm(iii)} the Jacobian of the map
$f=(\row fm):\Omega\arrow \Cm$ has rank $n$ at some point $z\in \Omega$. 
\end{theorem}

Repeating the proofs of Theorems~\ref{ST2} and~\ref{ST1} given in \cite{ST} with our Theorem~\ref{main1} in place of Theorem~\ref{STapprox} shows that Theorems~\ref{ST2} and~\ref{ST1} continue to hold with the hypothesis that $\Omega$ is an $L^\infty$-domain replaced by the hypothesis that $\Omega$ is open in $\mh$.  (See, however, Remark~\ref{correction} at the end of our paper).

%
%

\section{The Proofs}

We will need the following elementary lemma whose proof we include for completeness.

\begin{lemma}\label{trivial}
Let $\Sigma$ be a topological space and let $A$ be a supremum normed Banach algebra of bounded continuous complex-valued functions on $\Sigma$ that separates points and contains the constants.  Let $\row fm$ be functions in $A$.  Then the map $r:\maf\arrow\ma$ that sends each multiplicative linear functional on $\Afb$ to its restriction to $A$ is injective.  
\end{lemma}

\begin{proof}
By replacing the functions in $A$ and the functions $\row {\ol f}m$ by their continuous extensions to the closure of $\Sigma$ in $\ma$, we may assume without loss of generality that $\Sigma$ is compact and $A$ and $A[\row {\ol f}m]$ are uniform algebras on $\Sigma$.
Now suppose $\varphi_1$ and $\varphi_2$ are two multiplicative linear functionals on $\Afb$ whose restrictions to $A$ coincide.  Choose representing measures $\mu_1$ and $\mu_2$ on $\Sigma$ for $\varphi_1$ and $\varphi_2$, respectively.  Then for each $j=1,\ldots, m$, we have
$$\varphi_1(f_j)=\int \ol f_j\, d\mu_1=\ol{\int f_j\, d\mu_1}=\ol{\varphi_1(f_j)}=\ol{\varphi_2(f_j)}=\ol{\int f_j\, d\mu_2}=\int \ol f_j\, d\mu_2=\varphi_2(f_j).$$
Consequently, $\varphi_1=\varphi_2$.
\end{proof}

\begin{proof}[Proof of Theorems~\ref{main1} and~\ref{main2}]
Set $A=\Hinfty$, for the proof of Theorem~\ref{main1}, or $A=\Ao$, for the proof of Theorem~\ref{main2}. 
Set $B=A[\row {\ol f}m]$.  Let $\what B$ denote the uniform algebra on $\mb$ whose members are the Gelfand transforms of the functions in $B$.
By Lemma~\ref{trivial} we can regard $\mb$ as a subspace of $\ma$ by identifying each element of $\mb$ with its restriction to $A$.  Since $\Omega$ is open in $\mb$ and is contained in the subspace $\ma$, the set $\Omega$ is open in $\mb$ as well.  

We can regard $g$ as defined and continuous on all of $\Cn$ by considering $g$ to be identically zero on $\Cn\sm\Omega$.  
Note that $\Omega$, regarded as a subset of $\mb$, is closed in $\pi_B^{-1}(\Omega)$
(because it is the subset of $\pi_B^{-1}(\Omega)$ where the two continuous functions $\pi_B$ and the identity function agree).  Thus the closure of $\Omega$ in $\mb$ is contained in $\Omega \cup \pi_B^{-1}(\Cn\sm\Omega)$.  Since the function $g\circ\pi_B$ is identically zero on $\pi_B^{-1}(\Cn\sm\Omega)$, it follows that there is a well-defined continuous function $\tg$ on $\mb$ given by
$$
\tg(x) = 
\begin{cases} (g\circ\pi_B)(x) &\mbox{for\ } x\ \mbox{in\ the\ closure\ of\ $\Omega$ in $\mb$} \\
0 & \mbox{for\ } x\ \mbox{in\ } \mb\sm\Omega. 
\end{cases}
$$
By applying the Bishop antisymmetric decomposition (see \cite[Theorem~2.7.5]{Browder}, \cite[Theorem~II.13.1]{Gamelin}, or \cite[Theorem~12.1]{stout}),
we will show that $\tg$ is in $\what  B$.   It follows that $g$ is in $B$.

Let $E$ be a maximal set of antisymmetry for $\what B$. Since the real and imaginary parts of each of $\row fm$ lie in $B$, the set $E$ must be contained in a common level set of the functions $\row {\what f}m$.  Let $\row \lambda m$ denote the respective constant values of $\row {\what f}m$ on $E$.  By the definition of $\Omega_f$, each point of the set $L_\lambda=\{ z\in\Omega: \bigl(f_1(z),\ldots, f_m(z)\bigr)=(\row \lambda m)\}$ that is not in $\Omega_f$ is an isolated point of $L_\lambda$.  Because $\Omega$ is open in $\mb$, it follows that each point of $L_\lambda$ that is not in $\Omega_f$ is also an isolated point of the set 
$\widetilde L_\lambda=\{ z\in\mb: \bigl(\what f_1(z),\ldots, \what f_m(z)\bigr)=(\row \lambda m)\}$.  Since each maximal set of antisymmetry for a uniform algebra on its maximal ideal space is connected \cite[Remarks~12.7]{stout}, it follows that $E$ must be either a singleton set or else be contained in $(\mb\sm\Omega)\cup \Omega_f$.  Since $\tg$ is identically zero on $(\mb\sm\Omega)\cup \Omega_f$, we conclude that $\tg|_E$ is in $\what B|_E$.  Therefore, 
$\tg$ is in $\what B$ by the Bishop antisymmetric decomposition.
\end{proof}

\begin{proof}[Proof of Theorem~\ref{Chris}]
The basic idea is the same as in the preceding proof.  Set $B=\Hinfty[\ol f]$.  Let $\what B$ denote the uniform algebra on $\mb$ whose members are the Gelfand transforms of the functions in $B$.  Regard $\mb$ as a subspace of $\mh$ via Lemma~\ref{trivial}.

There is a continuous map $\pi_{\Hinfty}\arrow \ol\Omega$ that is the identity on $\Omega$ and takes $\mh\sm\Omega$ onto $b\Omega$.  When $\Omega\subset\C$ is bounded, $\pi_{\Hinfty}$ is just the Gelfand transform of $z$.  In the general case the definition of $\pi_\Hinfty$ is more complicated.  See \cite{local}.  In particular, $\Omega$ is open in $\mh$, and hence, in the subspace $\mb$.  Let $\pi$ be the restriction of $\pi_\Hinfty$ to $\mb$.

Each maximal set of antisymmetry for $\what B$ must be contained in a level set of $\what f$, and by \cite[Remarks~12.7]{stout} must be connected.  Since each level set of a nonconstant holomorphic function of one complex variable is discrete, and $\Omega$ is open in $\mb$, it follows that each maximal set of antisymmetry for $\what B$ is either a singleton set of else is contained in $\mh\sm\Omega=\pi^{-1}(b\Omega)$.  Invoking the Bishop antisymmetric decomposition, we conclude that the essential set for $\Hinfty[\ol f]$ is contained in $\pi^{-1}(b\Omega)$.  By \cite[Theorem~4.1]{Izzo2003}, which we quote below for the reader's convenience, it follows at once that $C(\ol \Omega)\subset \Hinfty[\ol f]$.
\end{proof}

\begin{theorem}[\cite{Izzo2003}]
Let $\Omega$ be an open set in the Riemann sphere, and
let $A$ be a uniformly closed algebra of bounded continuous complex-valued functions on $\Omega$.  If $\Omega\subset\C$ is bounded assume that $A\supset\Ao$, and if $\Omega$ is unbounded assume that $A\supset \Hinfty$.  Let $\E$ denote the essential set of $A$ regarded as a uniform algebra on $\ma$.  Then $A\supset C(\ol\Omega)$ if and only if $\E\subset\pi^{-1}(b\Omega)$.
\end{theorem}

When $\Omega$ is bounded, the proof of this theorem is rather easy.  The case of unbounded $\Omega$ is more difficult.

\begin{remark}\label{correction}
{\rm While the basic idea of the proof of Theorem~\ref{ST1} given in \cite{ST} is correct, there is an incorrect statement there in the proof of the implication (ii) implies (iii).  (The algebra generated by the functions $\row zn$ is not dense in $H^\infty(B)$ for $B$ an open ball in $\Cn$.)  We therefore repeat the proof of this implication avoiding that error.
Suppose that $\Hfb$ is dense in $L^p(\Omega)$ for some $1\leq p<\infty$.  Then applying \cite[Theorem~4.2 or Lemma~4.3]{ib} yields that for some point $z\in\Omega$ the differentials of the functions in the set $\Hinfty\cup\{\row {\ol f}m\}$ span a $2n$-dimensional vector space (the complexified cotangent space to $\Cn$).  Since the differential of every function in $\Hinfty$ lies in the $n$-dimensional space spanned by $dz_1,\ldots, dz_n$, it follows that the Jacobian of $f=(\row fm)$ has rank $n$ at $z$.}

\end{remark}

\section{Aknowlegement}
We thank Akaki Tikaradze for useful conversations.

\bibliographystyle{amsalpha}
\bibliography{gleason}

\end{document}